\documentclass{amsart}

\usepackage{enumerate}
\usepackage{amssymb}
\usepackage{graphicx}
\usepackage{MnSymbol}

\usepackage{amsthm}

\title[Indecomposable groups]{Freely indecomposable almost free groups with free abelianization}
\author{Samuel M. Corson}

\theoremstyle{definition}\newtheorem{theorem}{Theorem}

\theoremstyle{definition}

\numberwithin{theorem}{section}
\theoremstyle{definition}
\theoremstyle{definition}
\theoremstyle{definition}
\theoremstyle{definition}
\theoremstyle{definition}
\theoremstyle{definition}\newtheorem{remark}[theorem]{Remark}
\theoremstyle{definition}
\theoremstyle{definition}\newtheorem{lemma}[theorem]{Lemma}
\theoremstyle{definition}
\theoremstyle{definition}
\theoremstyle{definition}\newtheorem{definitions}[theorem]{Definitions}
\theoremstyle{definition}
\theoremstyle{definition}
\theoremstyle{definition}\newtheorem{construction}[theorem]{Construction}

\newcommand{\Len}{\operatorname{Len}}

\begin{document}

\address{Ikerbasque- Basque Foundation for Science and Matematika Saila, UPV/EHU, Sarriena S/N, 48940, Leioa - Bizkaia, Spain}

\email{sammyc973@gmail.com}
\thanks{This work was supported by ERC grant PCG-336983}

\begin{abstract}  For certain uncountable cardinals $\kappa$ we produce a group of cardinality $\kappa$ which is freely indecomposable, strongly $\kappa$-free, and whose abelianization is free abelian of rank $\kappa$.  The construction takes place in G\"odel's constructible universe $L$.  This strengthens an earlier result of Eklof and Mekler \cite{EM}.
\end{abstract}

\maketitle

\begin{section}{Introduction}

We produce examples of groups which exhibit some properties enjoyed by free groups but which in other ways are very far from being free.  We recall some definitions before stating the main result.  Given a group $G$ and subgroup $H\leq G$ we say $H$ is a \emph{free factor of $G$} provided there exists another subgroup $K\leq G$ such that $G= H * K$ in the natural way (that is- the map $H * K \rightarrow G$ induced by the inclusions of $H$ and $K$ is an isomorphism).  We call such a writing $G = H * K$ a \emph{free decomposition of $G$} and say that $G$ is \emph{freely indecomposable} provided there does not exist a free decomposition of $G$ via two nontrivial free factors.

Given a cardinal $\kappa$ we say $G$ is \emph{$\kappa$-free} if each subgroup of $G$ generated by fewer than $\kappa$ elements is a free group.  Historically a $\kappa$-free group of cardinality $\kappa$ is called \emph{almost free} \cite{H1}.  By the theorem of Nielsen and Schreier every free group is $\kappa$-free for every cardinal $\kappa$.  A subgroup $H$ of a $\kappa$-free group $G$ is \emph{$\kappa$-pure} if $H$ is a free factor of any subgroup $\langle H\cup X\rangle$ where $X \subseteq G$ is of cardinality $<\kappa$.  A $\kappa$-free group $G$ is \emph{strongly $\kappa$-free} provided each subset $X \subseteq G$ with $|X|<\kappa$ is included in a $\kappa$-pure subgroup of $G$ generated by fewer than $\kappa$ elements.

Let ZFC denote the Zermelo-Fraenkel axioms of set theory including the axiom of choice, and V = L denote the assertion that every set is constructible.  The theory ZFC + V = L is consistent provided ZFC is consistent \cite{G}.  The set theoretic concepts in the following statement will be reviewed in Section \ref{Proof} but the reader can, for example, let $\kappa$ be any uncountable succesor cardinal (e.g. $\aleph_1$, $\aleph_2$, $\aleph_{\omega +1}$):

\begin{theorem} \label{maintheorem} (ZFC + V = L)  Let $\kappa$ be uncountable regular cardinal that is not weakly compact.  There exists a group $G$ of cardinality $\kappa$ for which

\begin{enumerate}

\item $G$ is freely indecomposable;

\item $G$ is strongly $\kappa$-free;

\item the abelianization $G/G'$ is free abelian of cardinality $\kappa$.
\end{enumerate}
\end{theorem}

The hypotheses on the cardinal $\kappa$ cannot be dropped since a $\kappa$-free group of cardinality $\kappa$ must be free when $\kappa$ is singular or weakly compact (see respectively \cite{Sh} and \cite{E}).  A group as in the conclusion seems unusual since on a local level it is free, on a global level it is quite unfree and in fact indecomposable, but the abelianization is as decomposable as possible.  Theorem \ref{maintheorem} minus condition (3) was proved in \cite{EM} and the construction apparently does not have free abelianization; indeed, the proof that their groups are freely indecomposable involves abelianizing.  A non-free $\aleph_1$-free group of cardinality $\aleph_1$ which abelianizes to a free abelian group was produced by Bitton \cite{B} using only ZFC, and Theorem \ref{maintheorem} can be considered a constructible universe strengthening of his result.  The first construction of a non-free almost free group of cardinality $\aleph_1$ was given by Higman \cite{H1} without any extra set theoretic assumptions; a strongly $\aleph_1$-free group of cardinality $\aleph_1$ produced only from ZFC was given by Mekler \cite{M}.  The reader can find other results related to almost free (abelian) groups in such works as \cite{EMLarge}, \cite{MagSh}, \cite{Esmall}.

We note that is is not possible to produce a group $G$ of cardinality $\geq \kappa$ whose every subgroup of cardinality $\kappa$ satisfies conditions (1)-(3) of Theorem \ref{maintheorem}.  This is because for every uncountable locally free group $G$ there exists a free subgroup $H\leq G$ with $|H| = |G|$ \cite[Theorem 1.1]{N}.  Finally, we mention that the construction used in proving Theorem \ref{maintheorem} allows one, as in \cite{EM}, to construct $2^{\kappa}$ many pairwise non-isomorphic groups of this description using \cite{So}.
\end{section}

\begin{section}{Some Group Theoretic Lemmas}

The following appears as Lemma 1 in \cite{H1}:

\begin{lemma} \label{Kurosh}  If $K$ is a free factor of $G$ and $H$ is a subgroup of $G$ then $H \cap K$ is a free factor of $H$.
\end{lemma}

We will use the following construction in our proof of Theorem \ref{maintheorem}:

\begin{construction}\label{interestinginductioncase}  Suppose that we have a free group $F_a$ with free decomposition $F = F_0 * F_1 *  F_{2, a}$ with $F_1$ nontrivial and $F_2$ freely generated by $\{t_n\}_{n\in \omega}$.  Let $F_b$ be a free group with free decomposition $F_b = F_0 * F_1 * F_{2, b}$ where $F_{2, b}$ is freely generated by $\{z_n\}_{n\in \omega}$.  Let $y$ be an element of a free generating set for $F_1$.  Define $\phi: F_a \rightarrow F_b$ so that $\phi\upharpoonright F_0*F_1$ is identity and $\phi(t_n) = yz_{n+1}^{-1}z_nz_{n+1}$.  
\end{construction}

Property (iv) of the following lemma compares to \cite[Lemma 2.9 (3) \& (4)]{B}:

\begin{lemma}\label{basicfacts}  Let $F_{2, a, n} = \langle t_0, \ldots, t_{n-1}\rangle$.  The map $\phi$ satisfies the following:

\begin{enumerate}[(i)]

\item $\phi$ is a monomorphism;

\item $z_n\notin \phi(F_a)$ for all $n\in \omega$;

\item $\phi(F_a)$ is not a free factor of $F_b$ but $\phi(F_0*F_1*F_{2, a, n})$ is a free factor for every $n\in \omega$;

\item the equality $(\phi(F_a))' = \phi(F_a)\cap F_b'$ holds and the natural induced map $\overline{\phi}: F_a/F_a' \rightarrow F_b/F_b'$ is an isomorphism.
\end{enumerate}
\end{lemma}

\begin{proof}  Fix a possibly empty set of free generators $X$ for $F_0$ and a possibly empty set $Y$ such that the disjoint union $Y \sqcup \{y\}$ freely generates $F_1$.   Now $\phi(F_a)$ is the subgroup $\langle X \cup Y\cup \{y\} \cup \{yz_{n+1}^{-1}z_nz_{n+1}\}_{n\in \omega}\rangle = \langle X \cup Y\cup \{y\} \cup \{z_{n+1}^{-1}z_nz_{n+1}\}_{n\in \omega}\rangle$ and the generators  $X \cup Y\cup \{y\} \cup \{z_{n+1}^{-1}z_nz_{n+1}\}_{n\in \omega}$ freely generate $\phi(F)$ since they satisfy the Nielsen property (see \cite[I.2]{LS} and \cite[Example 2.8 (iii)]{H2}).  The set $X \cup Y\cup \{y\}\cup \{y^{-1}t_n\}_{n\in \omega}$ is a free generating set for $F_a$ and maps bijectively under $\phi$ to the free generating set $X \cup Y\cup \{y\} \cup \{z_{n+1}^{-1}z_nz_{n+1}\}_{n\in \omega}$ for $\phi(F_a)$.  Thus $\phi$ is a monomorphism and we have (i).  Moreover since $X \cup Y\cup \{y\} \cup \{z_{n+1}^{-1}z_nz_{n+1}\}_{n\in \omega}$ satisfies the Nielsen property it is clear that $z_n\notin \phi(F_a)$ for all $n\in \omega$ and we have (ii).

For (iii) we notice that $\langle\langle \phi(F_a)\rangle\rangle = F_b$ since the free generators listed for $F_b$ are conjugate in $F_b$ to free generators of $\phi(F_a)$.  Since $\phi(F_a)$ is a proper subgroup of $F_b$ this means that $\phi(F_a)$ cannot be a free factor of $F_b$.  On the other hand we have $\phi(F_0*F_1*F_{2, a, n})$ generated by $X \cup Y\cup \{y\} \cup \{yz_1^{-1}z_0z_1, \ldots, yz_n^{-1}z_{n-1}z_n\}$ and we claim that $X \cup Y\cup \{y\} \cup \{yz_1^{-1}z_0z_1, \ldots, yz_n^{-1}z_{n-1}z_n\} \cup \{z_n, z_{n+1}, \ldots\}$ is a free generating set for $F_b$.  It is plain that $\langle \{y\}\cup \{yz_1^{-1}z_0z_1, \ldots, yz_n^{-1}z_{n-1}z_n\} \cup \{z_n\}\rangle = \langle \{y\} \cup \{z_0, \ldots, z_n\}\rangle$ and since finitely generated free groups are Hopfian we get that $\{y\}\cup \{yz_1^{-1}z_0z_1, \ldots, yz_n^{-1}z_{n-1}z_n\} \cup \{z_n\}$ is a free generating set of the free factor $\langle \{y\} \cup \{z_0, \ldots, z_n\}\rangle$ of $F_b$.  Thus indeed $X \cup Y\cup \{y\} \cup \{yz_1^{-1}z_0z_1, \ldots, yz_n^{-1}z_{n-1}z_n\} \cup \{z_n, z_{n+1}, \ldots\}$ is a free generating set for $F_b$ and so $\phi(F_0*F_1*F_{2, a,  n}) = \langle X \cup Y\cup \{y\} \cup \{yz_1^{-1}z_0z_1, \ldots, yz_n^{-1}z_{n-1}z_n\}\rangle$ is a free factor of $F_b$ and we have shown (iii).

For condition (iv) certainly the inclusion $(\phi(F_a))' \subseteq \phi(F_a) \cap F_b'$ holds.  Moreover a word $w$ in $(X \cup Y\cup \{y\}\cup \{z_n\}_{n\in \omega})^{\pm 1}$ represents an element of $F_b'$ if and only if the sum of the exponents of each element in $X \cup Y\cup \{y\}\cup \{z_n\}_{n\in \omega}$ is $0$.  By treating each element of $\{z_{n+1}^{-1}z_nz_{n+1}\}_{n\in \omega}$ as an unreducing letter, a word in $(X \cup Y\cup \{y\} \cup \{z_{n+1}^{-1}z_nz_{n+1}\}_{n\in \omega})^{\pm 1}$ represents an element of $(\phi(F_a))'$ if and only if the sum of the exponents of each element of $X \cup Y\cup \{y\} \cup \{z_{n+1}^{-1}z_nz_{n+1}\}_{n\in \omega}$ is $0$.  This is clearly equivalent to having the sum of the exponents of each letter in $X \cup Y\cup \{y\}\cup \{z_n\}_{n\in \omega}$ be $0$, so we have $(\phi(F))' = \phi(F) \cap F'$.  Thus the map $\overline{\phi}: F_a/F_a' \rightarrow F_b/F_b'$ is injective.  Moreover since each element of $X \cup Y\cup \{y\}\cup \{z_n\}_{n\in \omega}$ is conjugate in $F_b$ to an element of $X \cup Y\cup \{y\} \cup \{z_{n+1}^{-1}z_nz_{n+1}\}_{n\in \omega}$ the map $\overline{\phi}$ is onto as well.
\end{proof}

We recall some notions for free products of groups (see \cite[IV.1]{LS}.  Suppose that we have a free product $L_0 *  L_1$.  We call the nontrivial elements in $L_0\cup L_1$ \emph{letters}.  Each element $g\in L_0 *  L_1$ can be expressed uniquely as a product of letters $g = h_0h_1\cdots h_{n-1}$ such that $h_i\in L_0$ if and only if $h_{i+1}\in L_1$ for all $0\leq i<n-1$ (this is the \emph{reduced} or \emph{normal form} of the element).  We call the number $n$ the \emph{length} of $g$, denoted $\Len(g)$.  Thus the identity element has length $0$ and a nontrivial element has length $1$ if and only if it is a letter.  Given a writing of an element of $L_0 * L_1$ as a product of nontrivial elements $h_0\cdots h_{n-1}$ in $L_0 \cup L_1$ it is easy to determine the normal form of the element by taking a consecutive pair $h_ih_{i+1}$ for which $h_i, h_{i+1}$ are both in $L_0$ or both in $L_1$ and performing the group multiplication in the appropriate group.  This either gives the trivial element, in which case we remove the pair $h_ih_{i+1}$ from the expression, or it gives a nontrivial element $g_i$ and we replace $h_ih_{i+1}$ with $g_i$ in the expression.  This process reduces the number of letters in the writing by at least $1$ every time and so the process must eventually terminate, and it terminates at the normal form.  We will generally consider an element of $L_0 * L_1$ as a word (the normal form) in the letters.  For words $w_0$ and $w_1$ in the letters we will use $w_0 \equiv w_1$ to represent that $w_0$ and $w_1$ are the same word letter-for-letter when read from left to right.

We say an element of $L_0 * L_1$ is \emph{cyclically reduced} if its reduced form is either of length $0$ or $1$ or begins with a letter in $L_j$ and ends with a letter in $L_{1-j}$.  A cyclically reduced element is of minimal length in its conjugacy class and if two cyclically reduced elements are conjugate to each other then the normal form of one is a cyclic shift of the other ($w$ is a cyclic shift of $u$ if we can write $w$ as a concatenation $w \equiv v_0v_1$ such that $u \equiv v_1v_0$).  Each element $g$ of $L_0 * L_1$ is conjugate to a cyclically reduced element $h$, and we call $\Len(h)$ the \emph{cyclic length} of $g$.

\begin{lemma}\label{spoilfreeprod}  If $\psi: F_a \rightarrow L_0 * L_1$ is a monomorphism such that $\psi(F_1) = \psi(F_a)\cap L_0$ and $\psi(F_0*F_{2, a}) = \psi(F_a) \cap L_1$ then there does not exist a monomorphism $\theta: F_b \rightarrow L_0*L_1$ for which $\theta\circ\phi = \psi$.
\end{lemma}

\begin{proof}  Suppose on the contrary that such a $\theta$ exists.  We will treat $F_a$ as a subgroup of $F_b$ since $\phi$ is a monomorphism and treat $F_a$ and $F_b$ as subgroups of $L_0 * L_1$ such that $F_1 = F_a\cap L_0$ and $F_0*F_{2, a} = F_a \cap L_1$.  We have that $t_n = yz_{n+1}^{-1}z_nz_{n+1}$ for all $n\in \omega$, and so $y^{-1}t_n = z_{n+1}^{-1}z_nz_{n+1}$.  Since $y^{-1}\in L_0$ and $t_n\in L_1$ we see that $y^{-1}t_n$ is cyclically reduced and $\Len(y^{-1}t_n) = 2$.  Therefore for each $n\in \omega$ we know $z_n$ is of cyclic length $2$, so $\Len(z_n) \geq 2$, and any cyclic reduction of $z_n$ must be a cyclic shift of $y^{-1}t_n$.  We claim that $\Len(z_n) \geq \Len(z_{n+1})+1$ for all $n\in \omega$.  This immediately gives $\Len(z_0) \geq \Len(z_n) +n$ for all $n\in \omega$ which is a contradiction.

It remains to prove that $\Len(z_n) \geq \Len(z_{n+1})+1$.  To economize on writing subscripts we will show that $\Len(z_0) \geq \Len(z_1) +1$ and the same proof will work for general $n$ by adding $n$ to the subscripts of $t_0, z_0, t_1, z_1$.  Suppose to the contrary that $\Len(z_0) \leq \Len(z_1)$.  We have $z_1y^{-1}t_0z_1^{-1} = z_0$.  Since $z_1$ is nontrivial it must end with a letter of $L_0$ or a letter of $L_1$.

\textbf{Case A: $z_1$ ends with a letter of $L_0$.}  In this case it must be that $z_1$ ends with $y$ since otherwise we readily see from the reduced form for $z_0 = z_1y^{-1}t_0z_1^{-1}$ that $\Len(z_0) = (2\Len(z_1) + 2) -1 = 2\Len(z_1) +1 > \Len(z_1)$ contrary to our assumption that $\Len(z_0) \leq \Len(z_1)$.  Also, the second-to-last letter of $z_1$ must be $t_0^{-1}$ since otherwise $L(z_0) = (2\Len(z_1) + 2) -3 = 2\Len(z_1) -1>\Len(z_1)$ since $\Len(z_1) \geq 2$.  Thus we may write $z_1$ as a reduced word $z_1 = w(t_0^{-1}y)^k$ with $k \geq 1$ and maximal.  Notice that $w$ is nonempty, for otherwise $z_1 = (t_0^{-1}y)^k$ and $z_0 = y^{-1}t_0$ and therefore $z_0$ and $z_1$ commute instead of generating a free subgroup of rank $2$.  The word $w$ must end with a letter from $L_0$ since $w(t_0^{-1}y)^k$ is reduced.  Moreover $\Len(w) \geq 2$ since otherwise $w = g \in L_0$ and $z_1$ is conjugate to $(t_0^{-1}y)^{k-1}t_0^{-1}yg$ and cyclically reducing this word cannot produce a cyclic shift of $y^{-1}t_1$.  Then the second-to-last letter of $w$ is an element from $L_1$.  If the last letter of $w$ is $y$ then the second-to-last letter of $w$ is not $t_0^{-1}$ by maximality of $k$ and we get

\begin{center}  $\Len(z_0) = (2\Len(z_1) +2) - 4k - 2 - 1 = 2\Len(z_1) -1 - 4k$.
\end{center}

\noindent If, on the other hand, the last letter of $w$ is $g\in L_0 \setminus \{1, y\}$ then we see that

\begin{center}  $\Len(z_0) = (2\Len(z_1) + 2) -4k -1 = 2\Len(z_1) - 4k +1$

\end{center}

\noindent so in either case we know $\Len(z_0) \geq 2\Len(z_1) -1 - 4k$.  Since we are assuming $\Len(z_1)\geq \Len(z_0)$ we have $\Len(z_1)\leq 4k +1$.

Write $z_1 = (y^{-1}t_0)^mu(t_0^{-1}y)^k$ with $m\geq 0$ maximal.  Certainly $u$ is nontrivial since otherwise $z_1$ would commute with $z_0$.  Also $u$ must begin and end with a letter from $L_0$.  If $m>k$ then $\Len(z_1) = 2m +2k +\Len(u) > 4k +1$.  If $k = m$ then since $4k +1 \geq \Len(z_1) = 2m +2k +\Len(u) = 4k + \Len(u)$ we get $\Len(u) = 1$.  Then $z_1$ is conjugate to an element of length $1$, contradicting the fact that $z_1$ is of cyclic length $2$.  Thus $k > m$.

If $k> m+1$ then by maximality of $m$ we conjugate $z_1$ to a word $u(t_0^{-1}y)^{k-m-1}t_0^{-1}y$ where $u$ might or might not start with $y^{-1}$ but if it starts with $y^{-1}$ the second letter of $u$ would not be $t_0$.  Thus whether or not $u$ starts with $y^{-1}$ we know that $u(t_0^{-1}y)^{k-m-1}t_0^{-1}y$ has cyclic length at least $3$ despite being conjugate to $z_1$, a contradiction.  Thus we know precisely that $k = m + 1$ and so $z_1$ is conjugate to $ut_0^{-1}y$.  If $u$ does start with $y^{-1}$, say $u \equiv y^{-1}v$ for some word $v$, then we get $z_1$ conjugate to $vt_0^{-1}$.  Since the cyclic length of $z_1$ is $2$ we know that $v$ is nonempty, must start with a letter from $L_1$ and that letter must not be $t_0$ by the maximality of $m$.  Then $v \equiv tv_0$ with $t\in L_1 \setminus \{1, t_0^{-1}\}$ and $v_0$ begins and ends with a letter in $L_0$.  Thus $z_1$ is conjugate to $v_0(t_0^{-1}t)$ which is a cyclically reduced word.  Since $z_1$ has cyclic length $2$ and cyclic reduction $y^{-1}t_1$ we obtain that $v_0 = y^{-1}$ and $t_0^{-1}t = t_1$.  But now

\begin{center}  $z_1 = (y^{-1}t_0)^{k-1}u(t_0^{-1}y)^k = (y^{-1}t_0)^{k-1}y^{-1}(t_0t_1)y^{-1}(t_0^{-1}y)^k$
\end{center}

\noindent so that $z_1$ is expressed as a product of elements in $F_a$, contradicting Lemma \ref{basicfacts} part (ii).  Therefore $u$ must start with some $g\in L_0 \setminus \{1, y^{-1}\}$, say $y \equiv gv_0$.  Now we conjugate $ut_0^{-1}y \equiv gv_0t_0^{-1}y$ to $v_0t_0^{-1}(yg)$ which is cyclically reduced (by considering $(yg)$ as a single letter in $L_0$).  But $v_0t_0^{-1}(yg)$ cannot possibly be a cyclic shift of $y^{-1}t_1$, regardless of whether $v$ is empty or not.  This finishes the proof of Case A.

\textbf{Case B: $z_1$ ends with a letter of $L_1$.}  The reasoning in this case follows that in the other case more or less and we give the sketch.  Arguing as before we see that $z_1$ must end with $t_0$ and the second-to-last letter must be $y^{-1}$.  Write $z_1 \equiv w(y^{-1}t_0)^k$ with $k\geq 1$ maximal.  As before, $w$ is nonempty.  Again we have $\Len(z_0) \geq 2\Len(z_1) - 1 -4k$ from which $\Len(z_1) \leq 4k + 1$.  Writing $z_1 \equiv (t_0^{-1}y)^mu(y^{-1}t_0)^k$ with $m\geq 0$ maximal we again see that $u$ is nonempty and that $k = m + 1$.  Also, $u$ must begin and end with letters from $L_1$.  Therefore $z_1$ is conjugate to $uy^{-1}t_0$.  If $u \equiv t_0^{-1}v$ for some $v$ then we get $z_1$ conjugate to $vy^{-1}$ and as $m$ was maximal the word $v$ starts with an element of $L_0$ which is not $y$.  Then $v\equiv gv_0$ with $g\in L_0\setminus \{1, y\}$ and $v_0$ nontrivial since $u$, and therefore $v$, must end in a letter from $L_1$.  Conjugating $vy^{-1}$ to the cyclically reduced word $v_0(y^{-1}g)$ we must have that this word is a cyclic shift of $y^{-1}t_1$.  Then $v_0 = t_1$ but $y^{-1}g$ cannot be $y^{-1}$ since $g$ was nontrivial.  Therefore it must be that $u \equiv tv$ for some $t\in L_1 \setminus \{1, t_0\}$.  Conjugating $uy^{-1}t_0 = tvy^{-1}t_0$ to the cyclically reduced word $vy^{-1}(t_0t)$, it must be that this word is a cyclic shift of $y^{-1}t_1$.  Then $v$ is empty and $t_0t = t_1$.  Therefore

\begin{center}  $z_1 = (t_0^{-1}y)^{k-1}u(y^{-1}t_0)^k = (t_0^{-1}y)^{k-1}(t_0^{-1}t_1)(y^{-1}t_0)^k$
\end{center}

\noindent and we have $z_1\in F_a$, which contradicts Lemma \ref{basicfacts} part (ii).  This completes the proof of Case B and of the lemma.
\end{proof}

\end{section}

\begin{section}{Proof of the Main Theorem}\label{Proof}

The groups that we produce for Theorem \ref{maintheorem} will follow the induction used in \cite[Theorem 2.2]{EM}, using Construction \ref{interestinginductioncase} at the key stages.  We will use combinatorial principles which follow from ZFC + V = L to rule out any possible free decomposition while ensuring strong $\kappa$-freeness and free abelianization.  We first review some concepts from set theory.

\begin{definitions}(see \cite{Jec}) Recall that a cardinal number is naturally considered as an ordinal number which cannot be injected into a proper initial subinterval of itself.  A subset $E$ of an ordinal $\alpha$ is \emph{bounded in $\alpha$} if there exists $\beta< \alpha$ which is an upper bound on $E$.  The \emph{cofinality} of an ordinal is the least cardinal $\kappa$ for which there exists an unbounded $E \subseteq \alpha$ of cardinality $\kappa$.  An infinite cardinal $\kappa$ is \emph{regular} if the cofinality of $\kappa$ is $\kappa$.  A subset $C$ of ordinal $\alpha$ is \emph{club} if it is unbounded in $\alpha$ and closed under the order topology in $\alpha$.  The intersection of two club sets in an uncountable regular cardinal is again a club set.  A subset $E$ of ordinal $\alpha$ is \emph{stationary} if it has nonempty intersection with every club subset of $\alpha$.  The intersection of a club set and a stationary set in a regular cardinal is again stationary.  We mention that weakly compact cardinals are inaccessible and it is therefore consistent to assume that a universe of set theory does not contain any (see \cite[Chapters 9, 17]{Jec}).
\end{definitions}

The following is a theorem of Solovay (see \cite{So} or \cite[Theorem 8.10]{Jec}):

\begin{theorem}\label{Solovay}  If $\kappa$ is an uncountable regular cardinal then each stationary subset of $\kappa$ can be decomposed as the disjoint union of $\kappa$ many stationary subsets of $\kappa$.
\end{theorem}

We quote Jensen's $\diamond_{\kappa}(E)$ principle (see \cite[Lemma 6.5]{Jen} or \cite[27.16]{Jec}), remark an easy consequence (see \cite[page 97]{EM}) and quote one more result of Jensen (see \cite[Theorem 5.1]{Jen}, \cite[Theorem 1.3]{EM}):

\begin{theorem}\label{Jensen1}  (ZFC + V = L)  If $\kappa$ is an uncountable regular cardinal and $E \subseteq \kappa$ is stationary in $\kappa$ there exists a sequence $\{S_{\alpha}\}_{\alpha\in E}$ such that $S_{\alpha} \subseteq \alpha$ and for any $J\subseteq \kappa$ the set $\{\alpha\in E\mid J\cap \alpha = S_{\alpha}\}$ is stationary in $\kappa$.
\end{theorem}

\begin{remark}\label{twofromone}  From a sequence given by Theorem \ref{Jensen1} one obtains a sequence of ordered pairs $\{(T_{\alpha}^0, T_{\alpha}^1)\}_{\alpha \in E}$ for which $T_{\alpha}^0, T_{\alpha}^1 \subseteq \alpha$ and given any subsets $L_0, L_1\subseteq \kappa$ the set $\{\alpha\in E\mid L \cap \alpha = T_{\alpha}^0\text{ and }L_1 \cap \alpha = T_{\alpha}^1\}$ is stationary in $\kappa$.  To see this we give the product $\{0, 1\}\times \kappa$ the lexicographic order.  There is an order isomorphism $f: \{0, 1\}\times \kappa \rightarrow \kappa$ under which $f(0, \alpha) = \alpha$ for every limit ordinal $\alpha<\kappa$.  Let $f_i:\kappa \rightarrow \kappa$ be given by $f_i(\alpha) = f(i, \alpha)$ for each $i\in \{0, 1\}$, so $f_0$ has disjoint image from $f_1$.  Notice that for each limit ordinal $\alpha<\kappa$ and $X\subseteq \kappa$ we have $f_i(X\cap \alpha) = f_i(X) \cap \alpha$.  Letting $E_0$ be the intersection of $E$ with the set of limit ordinals below $\kappa$ we have that $E_0$ is stationary as the intersection of a stationary with a club.  For $\alpha\in E_0$ define $T_{\alpha}^0 = f_0^{-1}(S_{\alpha})$ and $T_{\alpha}^1 = f_1^{-1}(S_{\alpha})$.  Given $L_0, L_1\subseteq \kappa$ we let $J = f_0(L_0) \sqcup f_1(L_1)$ and notice that

\begin{center}  $\{\alpha\in E_0\mid L_0\cap \alpha = T_{\alpha}^0 \text{ and }L_1\cap \alpha = T_{\alpha}^1\}$

$= \{\alpha\in E_0\mid f_0(L_0\cap \alpha) \sqcup f_1(L_1\cap \alpha) = S_{\alpha}\}$

$= \{\alpha \in E_0\mid (f_0(L_0) \cap \alpha) \sqcup (f_1(L_1)\cap \alpha) = S_{\alpha}\}$

$= \{\alpha \in E_0\mid J \cap \alpha = S_{\alpha}\}$

\end{center}

\noindent is stationary in $\kappa$.  Letting $T_{\alpha}^0 = \emptyset = T_{\alpha}^1$ for $\alpha \in E \setminus E_0$ gives the desired sequence of ordered pairs. 

\end{remark}

\begin{theorem}\label{Jensen2} (ZFC+ V = L)  If $\kappa$ is an uncountable regular cardinal which is not weakly compact then there is a stationary subset $E \subseteq \kappa$ for which each element of $E$ is a limit ordinal of cofinality $\omega$ and for each limit ordinal $\alpha<\kappa$ the set $E\cap \alpha$ is not stationary in $\alpha$.
\end{theorem}

\begin{proof}(of Theorem \ref{maintheorem})  Let $\kappa$ be as in the hypotheses of Theorem \ref{maintheorem}.  We inductively define a group structure on $\kappa$ and this will serve as our group $G$.  Let $\gamma_{0} = \omega$, $\gamma_{\alpha + 1} = \gamma_{\alpha} + \gamma_{\alpha}$ and $\gamma_{\beta} = \bigcup_{\alpha<\beta} \gamma_{\alpha}$ for limit ordinal $\beta<\kappa$.  The set $\{\gamma_{\alpha}\}_{\alpha<\kappa}$ is obviously a club set in $\kappa$.  We will define a group structure on each $\gamma_{\alpha}$ so that $\gamma_{\alpha}$ will be a subgroup of $\gamma_{\beta}$ whenever $\alpha<\beta$ and thus the group structure on $\kappa = \bigcup_{\alpha<\kappa}\gamma_{\alpha}$ will be well defined.

Select stationary $E_0 \subseteq \kappa$ satisfying the conclusion of Theorem \ref{Jensen2} and such that $\kappa\setminus E_0$ is stationary in $\kappa$ (using Theorem \ref{Solovay}).  It is easy to check that $E = \{\gamma_{\alpha}\}_{\alpha \in E_0}$ is also stationary in $\kappa$.  Let $\{(T_{\gamma_{\alpha}}^0, T_{\gamma_{\alpha}}^1)\}_{\gamma_{\alpha}\in E}$ be a sequence as in Remark \ref{twofromone}.

The following properties will hold on the groups for all $\alpha<\beta<\kappa$:

\begin{enumerate}[(i)]

\item $\gamma_{\alpha}$ is free of infinite rank;

\item $\gamma_{\alpha}$ is a proper subgroup of $\gamma_{\beta}$;

\item $\gamma_{\alpha}$ is a free factor of $\gamma_{\beta}$ if and only if $\gamma_{\alpha}\notin E$;

\item $\gamma_{\alpha}' = \gamma_{\beta}' \cap \gamma_{\alpha}$.
\end{enumerate}

We let $\gamma_0$ be free of countably infinite rank.  Suppose that we have defined the group structure on $\gamma_{\alpha}$ for all $\alpha<\beta <\kappa$ such that the above conditions hold.  If $\beta = \delta + 1$ and $\gamma_{\delta}\notin E$ then we give $\gamma_{\beta}$ the group structure obtained by any bijection $f: \gamma_{\beta} \rightarrow \gamma_{\delta} * \mathbb{Z}$ such that $f\upharpoonright \gamma_{\alpha}$ is the identity map.  Such a bijection exists since $|\gamma_{\beta}\setminus \gamma_{\delta}| = |\gamma_{\delta}|$.  Conditions (i), (ii) and (iv) obviously hold.  Notice also that for $\alpha<\beta$ we have $\gamma_{\alpha}$ a free factor of $\gamma_{\beta}$ if and only if $\gamma_{\alpha}$ is a free factor of $\gamma_{\delta}$ (this uses Lemma \ref{Kurosh}) and so (iii) also holds.

If $\beta = \delta + 1$ and $\gamma_{\delta}\in E$ then we consider two subcases.  Firstly suppose that there exists a strictly increasing sequence $\{\alpha_n\}_{n\in \omega}$ for which

\begin{itemize}

\item $\gamma_{\alpha_n} \notin E$;

\item $\bigcup_{n\in\omega}\alpha_n = \delta$;

\item $\gamma_{\alpha_n} = (\gamma_{\alpha_n}\cap T_{\gamma_{\delta}}^0 ) *  (\gamma_{\alpha_n}\cap T_{\gamma_{\delta}}^0)$ with both free factors nontrivial for all $n\in\omega$.

\end{itemize}

\noindent  From this it follows immediately that $\gamma_{\delta} = (\gamma_{\delta}\cap T_{\gamma_{\delta}}^0 ) *  (\gamma_{\delta}\cap T_{\gamma_{\delta}}^1)$.  Since $\{\gamma_{\alpha_n}\}_{n\in\omega}$ is a strictly increasing sequence of sets we know for each $n\in \omega$ that it is either the case that $\gamma_{\alpha_{n+1}}\cap T_{\gamma_{\delta}}^0 \supsetneq \gamma_{\alpha_n}\cap T_{\gamma_{\delta}}^0$ or that  $\gamma_{\alpha_{n+1}}\cap T_{\gamma_{\delta}}^1 \supsetneq \gamma_{\alpha_n}\cap T_{\gamma_{\delta}}^1$.  By choosing a subsequence we can assume without loss of generality that $\gamma_{\alpha_{n+1}}\cap T_{\gamma_{\delta}}^0 \supsetneq \gamma_{\alpha_n}\cap T_{\gamma_{\delta}}^0$ for all $n\in \omega$.  Since each $\gamma_{\alpha_n}\notin E$ we know by our induction that $\gamma_{\alpha_n}$ is a free factor of $\gamma_{\delta}$ and also of $\gamma_{\alpha_{n+1}}$.  Then by Lemma \ref{Kurosh} we know that $\gamma_{\alpha_n}\cap T_{\gamma_{\delta}}^0$ is a free factor of $\gamma_{\delta}\cap T_{\gamma_{\delta}}^0$ and of $\gamma_{\alpha_{n+1}}\cap T_{\gamma_{\delta}}^0$.  Inductively select a free basis $Q$ of $\gamma_{\delta}\cap T_{\gamma_{\delta}}^0$ such that $Q\cap \gamma_{\alpha_n}\cap T_{\gamma_{\delta}}^0$ is a free generating set for $\gamma_{\alpha_n}\cap T_{\gamma_{\delta}}^0$ for each $n\in \omega$.  For each $n\in \omega$ select $t_n \in Q\cap \gamma_{\alpha_{n+1}}\cap T_{\gamma_{\delta}}^0\setminus\gamma_{\alpha_n}$ and let $X = Q \setminus \{t_n\}_{n\in \omega}$.  Let $\{y\}\sqcup Y$ be a free generating set for $\gamma_{\delta}\cap T_{\gamma_{\delta}}^1$.  Let $\gamma_{\beta}$ be a group freely generated by $X \cup \{z_n\}_{n\in \omega} \cup \{y\} \cup Y$ such that the inclusion map $\iota: \gamma_{\delta} \rightarrow \gamma_{\beta}$ is the map $\phi$ from Construction \ref{interestinginductioncase}.

Certainly conditions (i) and (ii) hold, and condition (iv) holds by Lemma \ref{basicfacts} (iv).  Also we know $\gamma_{\delta}$ is not a free factor of $\gamma_{\beta}$ by Lemma \ref{basicfacts} (iii).  Notice that for each $n\in \omega$ we have $\langle X \cup \{t_0, \ldots, t_{n-1}\}\cup \{y\} \cup Y\rangle$ is a free factor of $\gamma_{\beta}$ by Lemma \ref{basicfacts} (iii).  For each $n\in \omega$ we know $\gamma_{\alpha_n}$ is a free factor of $\langle X \cup \{t_0, \ldots, t_{n-1}\}\cup \{y\} \cup Y\rangle$ and for each $\alpha<\delta$ there exists $n\in \omega$ for which $\alpha<\alpha_n$.  It follows by Lemma \ref{Kurosh} that for $\alpha<\delta$ we have $\gamma_{\alpha}$ a free factor of $\gamma_{\beta}$ if and only if $\gamma_{\alpha}\notin E$, and condition (iii) holds.

On the other hand suppose $\beta = \delta + 1$ and $\gamma_{\delta}\in E$ and no such increasing sequence $\{\alpha_n\}_{n\in \omega}$ exists.  Since $\delta \in E_0$ is of cofinality $\omega$ and $E_0\cap \delta$ is not stationary in $\delta$ we may select a strictly increasing sequence $\alpha_n \notin E_0$ such that $\bigcup_{n\in \omega}\alpha_n = \delta$.  Then $\gamma_{\alpha_n}\notin E$.  As each $\gamma_{\alpha_n}$ is a free factor of $\gamma_{\delta}$ and the $\alpha_n$ are strictly increasing we may select by induction a free generating set $Q$ for $\gamma_{\delta}$ such that $Q\cap \gamma_{\alpha_n}$ is a free generating set for $\gamma_{\alpha_n}$.  Pick $y\in \gamma_{\alpha_0}\cap Q$ and $t_n\in \gamma_{\alpha_{n+1}} \cap Q \setminus \gamma_{\alpha_n}$ and letting $X = \emptyset$ and $Y = Q \setminus (\{t_n\}_{n\in \omega} \cup \{y\})$ we let $\gamma_{\beta}$ be a group freely generated by $X \cup \{z_n\}_{n\in \omega} \cup \{y\} \cup Y$ such that the inclusion map $\iota: \gamma_{\delta} \rightarrow \gamma_{\beta}$ is the map $\phi$ from Construction \ref{interestinginductioncase}.  The check that the induction conditions still hold is as in the other subcase.

When $\beta<\kappa$ is a limit ordinal the binary operation on $\gamma_{\beta} = \bigcup_{\alpha<\beta}\gamma_{\alpha}$ is defined by that on the $\gamma_{\alpha}$ for $\alpha<\beta$.  By how $E_0$ was chosen we know $E_0 \cap \beta$ is not stationary in $\beta$ and so we select a club set $C \subseteq \beta$ for which $C \cap E_0 = \emptyset$.  By induction we know that for $\alpha, \delta \in C$ with $\alpha<\delta$ we have $\gamma_{\alpha}$ is a proper free factor of $\gamma_{\delta}$ and since $C$ is closed we have $\gamma_{\delta} = \bigcup_{\alpha<\delta, \alpha\in C}\gamma_{\alpha}$ for any $\delta\in C$ which is a limit under the ordering of $\kappa$ restricted to $C$.  Then by induction on $C$ we can select a free generating set $Q$ for $\gamma_{\beta} = \bigcup_{\alpha\in C} \gamma_{\alpha}$ for which $Q \cap \gamma_{\alpha}$ is a free generating set for $\gamma_{\alpha}$ for each $\alpha \in C$.  Then (i) and (ii) hold.  For $\alpha<\beta$ it is clear by Lemma \ref{Kurosh} that $\gamma_{\alpha}$ is a free factor of $\gamma_{\beta}$ if and only if $\gamma_{\alpha}$ is a free factor of $\gamma_{\delta}$ for some $\delta\in C$ with $\delta > \alpha$, and since $\gamma_{\delta}$ is a free factor of $\gamma_{\beta}$ for each $\delta\in C$ condition (iii) holds.  Condition (iv) follows by induction since $\gamma_{\beta}' = \bigcup_{\alpha<\beta} \gamma_{\alpha}'$.  This completes the construction of the group structure on $\kappa$.

We verify that conditions (1)-(3) of the statement of Theorem \ref{maintheorem} hold.  Imagine for contradiction that $\kappa = L_0 * L_1$ with $L_0, L_1$ nontrivial subgroups of $\kappa$.  Letting $C = \{\alpha<\kappa\mid \gamma_{\alpha} = (\gamma_{\alpha}\cap L_0) * (\gamma_{\alpha} \cap L_1)\}$ it is straightforward to verify that $C$ is club in $\kappa$.  Since the set $\kappa \setminus E_0$ is stationary in $\kappa$ we know $D =  C \setminus E_0$ is stationary and therefore unbounded in $\kappa$.  Then the closure $\overline{D}$ is club in $\kappa$, and so is $\{\gamma_{\alpha}\}_{\alpha\in \overline{D}}$.  Then there exists $\gamma_{\delta} \in E$ with $\delta\in \overline{D}$ and $L_0 \cap \gamma_{\delta} = T_{\gamma_{\delta}}^0$ and $L_1 \cap \gamma_{\delta} = T_{\gamma_{\delta}}^1$.  Then $\delta \in E_0$ and so $\delta\in E_0\cap \overline{D}$.  As $\delta \in E_0$ we know that $\delta$ has cofinality $\omega$ in $\kappa$.  Certainly $\delta \notin D = C \setminus E_0$ and so there exists a strictly increasing sequence $\{\alpha_n\}_{n\in \omega}$ with $\alpha_n \in C \setminus E_0$ such that $\bigcup_{n\in \omega} \alpha_n = \delta$.  Then 

\begin{center}
$\gamma_{\alpha_n} = (\gamma_{\alpha_n}\cap L_0) * (\gamma_{\alpha_n} \cap L_1)$

$=  (\gamma_{\alpha_n}\cap T_{\gamma_{\delta}}^0) * (\gamma_{\alpha_n} \cap T_{\gamma_{\delta}}^1)$.
\end{center}

\noindent By our construction we know that $\gamma_{\delta}$ includes into $\gamma_{\delta +1}$ in such a way that $\gamma_{\delta + 1}$ is not a subgroup of $L_0 * L_1$ (using Lemma \ref{spoilfreeprod}), and this is a contradiction.  Thus $\kappa$ is freely indecomposable and we have part (1).

For part (2) we let $X\subseteq \kappa$ with $|X|<\kappa$.  By the regularity of $\kappa$ select $\alpha<\kappa$ large enough that $\gamma_{\alpha}\supseteq X$.  Notice that $\gamma_{\alpha+1}\notin E$.  Any subgroup $H$ of $\kappa$ with $|H|<\kappa$ and $H\geq \gamma_{\alpha +1}$ satisfies $H \leq \gamma_{\beta}$ for some $\beta >\alpha$ by the regularity of $\kappa$.  Since $\gamma_{\alpha+1}$ is a free factor of $\gamma_{\beta}$ by our construction, we have by Lemma \ref{Kurosh} that $\gamma_{\alpha + 1}$ is a free factor of $H$.  Thus $\gamma_{\alpha +1}$ is $\kappa$-pure and the group $\kappa$ is strongly $\kappa$-free and we have verified (2).

For part (3) we notice that since $\gamma_{\alpha}' = \gamma_{\alpha}\cap \gamma_{\beta}'$ for each $\alpha < \beta$ and since $\kappa' = \bigcup_{\beta<\kappa}\gamma_{\beta}'$ the equality $\gamma_{\alpha}' = \gamma_{\alpha}\cap\kappa'$ holds for all $\alpha<\kappa$.  In particular the homomorphism induced by the inclusion map $\gamma_{\alpha}/\gamma_{\alpha}' \rightarrow \kappa/\kappa'$ is injective and so the abelianization of $\kappa$ is the increasing union of free abelian subgroups $\gamma_{\alpha}/\gamma_{\alpha}'$.  In our construction, when $\beta = \delta + 1$ and $\gamma_{\delta} \in E$ the map induced by inclusion $\gamma_{\delta}/\gamma_{\delta}' \rightarrow \gamma_{\beta}/\gamma_{\beta}'$ is an isomorphism by Lemma \ref{basicfacts} (iv).  When $\beta = \delta+1$ and $\gamma_{\delta} \notin E$ we have $\gamma_{\delta}/\gamma_{\delta}'$ as a proper direct summand of $\gamma_{\beta}/\gamma_{\beta}'$.  As well we have $\gamma_{\beta}/\gamma_{\beta}' = \bigcup_{\alpha<\beta}\gamma_{\alpha}/\gamma_{\alpha}'$ for limit $\beta<\kappa$.  Thus we may inductively select a free abelian basis for the abelianization of $\kappa$, and the free basis will be of cardinality $\kappa$ since there are $\kappa$ many $\beta$ for which $\beta = \delta + 1$ with $\gamma_{\delta}\notin E$.  We have verified (3) and finished the proof of the theorem.
\end{proof}

\end{section}

\end{document}